\documentclass[12pt,a4paper]{amsart}

\usepackage[a4paper,centering,margin=2.5cm]{geometry}

\usepackage{color,url}
\usepackage{amssymb}
\usepackage{graphicx}
\usepackage{enumerate}

\title{Connected covering numbers}
\dedicatory{Dedicated to the memory of Michel Las Vergnas}
\author[J. Chappelon \and K. Knauer \and L.P. Montejano \and J.L. Ram\'irez Alfons\'in]{Jonathan Chappelon$^{\star}$ \and Kolja Knauer \and Luis Pedro Montejano \and Jorge Luis Ram\'irez Alfons\'in}
\thanks{The second and forth authors were supported by the ANR TEOMATRO grant ANR-10-BLAN 0207}
\thanks{The third author was supported by CONACYT}
\thanks{$^\star$ Corresponding Author: Phone: +33-467144166. Email: jonathan.chappelon@um2.fr}
\address{Universit\'{e} Montpellier 2, Institut de Math\'{e}matiques et de Mod\'{e}lisation de Montpellier, Case Courrier 051, Place Eug\`{e}ne Bataillon, 34095 Montpellier Cedex 05, France.}
\email{jonathan.chappelon@um2.fr}
\email{kolja.knauer@googlemail.com}
\email{lpmontejano@gmail.com}
\email{jramirez@um2.fr}
\date{January 11, 2015}

\theoremstyle{plain}
\newtheorem{theorem}{Theorem}[section]

\newtheorem{proposition}[theorem]{Proposition}

\newtheorem{claim}[theorem]{Claim}

\theoremstyle{definition}

\newtheorem{example}[theorem]{Example}
\newtheorem{conjecture}[theorem]{Conjecture}

\newtheorem{question}[theorem]{Question}

\theoremstyle{remark}

\newtheorem{remarks}[theorem]{Remarks}

\newtheorem{case}{Case}

\newcommand{\C}{\mathrm{C}}
\newcommand{\Su}{\mathrm{S}}
\newcommand{\N}{\mathrm{N}}
\newcommand{\CC}{\mathrm{CC}}
\newcommand{\CT}{\mathrm{CT}}
\newcommand{\T}{\mathrm{T}}

\renewcommand{\le}{\leqslant}
\renewcommand{\leq}{\leqslant}
\renewcommand{\ge}{\geqslant}
\renewcommand{\geq}{\geqslant}

\begin{document}

\begin{abstract}
A connected covering is a design system in which the corresponding {\em block graph} is connected. The minimum size of such coverings are called {\em connected coverings numbers}. In this paper, we present  various formulas and bounds for several parameter settings for these numbers. We also investigate results in connection with {\em Tur\'an systems}. Finally, a new general upper bound, improving an earlier result, is given. The latter is used to improve upper bounds on a question concerning oriented matroid due to Las Vergnas. 
\end{abstract}

\maketitle

\section{Introduction}

Let $n,k,r$ be positive integers such that $n\ge k\ge r\ge 1$. A \textit{$(n,k,r)$-covering} is a family $\mathcal{B}$ of $k$-subsets of $\{1,\ldots,n\}$, called \textit{blocks}, such that each $r$-subset of $\{1,\ldots,n\}$ is contained in at least one of the blocks. The number of blocks is the covering's \textit{size}. The minimum size of such a covering is called the \textit{covering number} and is denoted by $\C(n,k,r)$. Given a $(n,k,r)$-covering $\mathcal{B}$, its graph $G(\mathcal{B})$ has $\mathcal{B}$ as vertices and two vertices are joined if they have one $r$-subset in common. We say that a $(n,k,r)$-covering is \textit{connected} if the graph $G(\mathcal{B})$ is connected. The minimum size of a connected $(n,k,r)$-covering is called the \textit{connected covering number} and is denoted by $\CC(n,k,r)$.

\begin{figure}[htb] 
 \includegraphics[width=.5\textwidth]{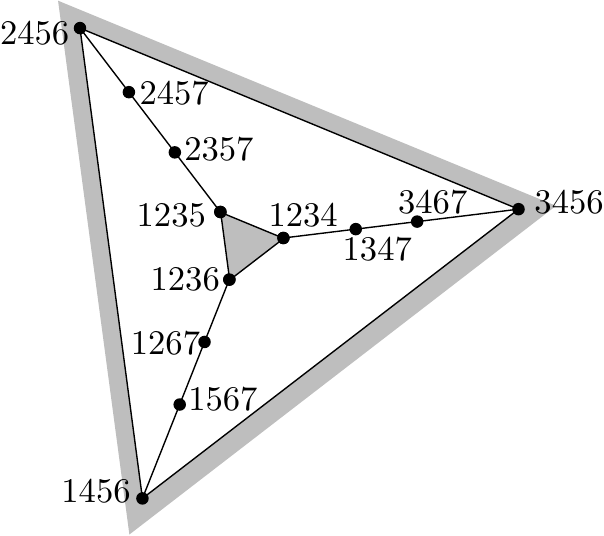}
 \caption{A connected $(7,4,3)$-covering with $12$ blocks.}
 \label{fig:CC743}
\end{figure}

The graph corresponding to a connected $(7,4,3)$-covering can be nicely illustrated as shown in Figure~\ref{fig:CC743}.
\par In this paper, we mainly pay our attention to coverings when $k=r+1$ and thus,  we will denote $\C(n,r+1,r)$ (resp. $\CC(n,r+1,r)$) by $\C(n,r)$ (resp. by $\CC(n,r)$) for short. The original motivation to study  $\CC(n,r)$ comes from the following question posed by Las Vergnas.

\begin{question}\label{quest:motiv}
Let $U_{r,n}$ be the rank $r$ uniform matroid on $n$ elements. What is the smallest number $s(n,r)$ of circuits of $U_{r,n}$, that uniquely determines all orientations of $U_{r,n}$? That is, whenever two uniform oriented matroids coincide on these circuits they must be equal.
\end{question}

In \cite{For-98}, Forge and Ram\'irez Alfons\'in introduced the notion of connected coverings and proved that 
\begin{equation}\label{l1}
s(n,r)\le \CC(n,r).
\end{equation}
The latter was then used to improve the best upper bound, $s(n,r)\le \binom{n-1}{r}$, known at that time due to Hamidoune and Las Vergnas \cite{Ham-86}; see also \cite{For-02} for related results.  

\smallskip

It turns out that $s(n,r)$ is also closely related to $\C(n,r)$. Indeed, by using results in \cite{For-98,For-02} it can be proved that
\begin{equation}\label{prop:motiv}
\C(n,r)\leq s(n,r).
\end{equation}

A proof (needing some oriented matroid notions and thus lying slightly out of scope of this paper) of a more general version of the above inequality can be found in \cite{Cha-13}.

Covering designs have been the subject of an enormous amount of research papers (see~\cite{Gor-95} for many upper bounds and~\cite{Sid-95} for a survey in the dual setting of Tur\'an-systems). Although the construction of block design is often very elusive and the proof of their existence is sometimes tough, here, we will be able to present explicit constructions yielding exact values and bounds for,$\C(n,r)$ and $\CC(n,r)$ for infinitely many cases. 
The study of $\C(n,r)$ and $\CC(n,r)$ seems to be interesting not only for Design Theory but also, in view of Equations~\eqref{prop:motiv} and~\eqref{l1}, for the implications on the behavior of $s(n,r)$ in Oriented Matroid Theory. 
This relationship was already remarked in~ \cite[Theorem 4.1]{For-98} where it was proved that  $\CC(n,r)\le 2\C(n,r).$ The latter can be slightly improved as follows
\begin{equation}\label{la1}
\CC(n,r)\le 2\C(n,r)-1,
\end{equation}
since the graph $G$ associated to a covering with $\C(n,r)$ blocks (and thus with $|V(G)|=\C(n,r)$) can be made connected by adding at most $\C(n,r)-1$ extra vertices (blocks), obtaining a graph corresponding to a $(n,r+1,r)$-connected covering with at most $2\C(n,r)-1$ blocks.

\medskip

Many interesting variants of Question~\ref{quest:motiv} can be investigated. For instance, for non-uniform (oriented) matroids (graphic, representable, etc.) and by varying the notion of what \emph{determine} means (up to orientations, bijections, etc.).  These (and other) variants are treated in another paper (see ~\cite{Cha-13}).

\medskip

This paper is organized as follows. In the next section, we recall some basic definitions and results concerning (connected) coverings and its connection with {\em Tur\'an systems} needed for the rest of the paper. In Section~\ref{ref:small}, we investigate connected covering numbers when the value $r$ is either small or close to $n$. Among other results, we give the exact value for $\CC(n,2)$ (Theorem~\ref{thm:CCn32}), for  $\CC(n,3)$ for $n\le 12$ (Theorem~\ref{prop:Cn43}) and for $\CC(n,n-3)$ (Theorem~\ref{thm:Turan1}). A famous conjecture of Tur\'an and its connection with our results is also discussed.  In Section~\ref{sec:gen}, we present a general upper bound for $\CC(n,r)$ (Theorem~\ref{th:gen}) allowing us to improve the best known upper bound for $s(n,r)$. We end the paper by discussing some asymptotic results in Section~\ref{sec:asym}.

\section{Basic results}

Let $n,m,p$ be positive integers such that $n\ge m\ge p$. A \textit{$(n,m,p)$-Tur\'an-system} is a family $\mathcal{D}$ of $p$-subsets of $\{1,\ldots,n\}$, called \textit{blocks}, such that each $m$-subset of $\{1,\ldots,n\}$ contains at least one of the blocks. The number of blocks is the \textit{size} of the Tur\'an-system. The minimum size of such a covering is called the \textit{Tur\'an Number} and is denoted by $\T(n,m,p)$. Given a $(n,m,p)$-Tur\'an-system $\mathcal{D}$, with $0\leq 2p-m\leq p$, its graph $G(\mathcal{D})$ has as vertices $\mathcal{D}$ and two vertices are joined if they have one $2p-m$-subset in common. We say that a $(n,m,p)$-Tur\'an-system with $0\leq 2p-m\leq p$ is \textit{connected} if $G(\mathcal{D})$ is connected. 

\smallskip

The minimum size of a connected $(n,m,p)$-Tur\'an-system is the \textit{connected Tur\'an Number} and is denoted by $\CT(n,m,p)$.  By applying set complement to blocks, it can be obtained that 
\begin{equation}\label{obs:dual1}
\C(n,k,r)=\T(n,n-r,n-k).
\end{equation}
Moreover, if  $0\leq n-2k+r\leq n-k$ then
\begin{equation}\label{obs:dual2}
\CC(n,k,r)=\CT(n,n-r,n-k).
\end{equation}
Note that the precondition for~\eqref{obs:dual2} is fulfilled if $k=r+1$.

Most of the papers on coverings consider $n$ large compared with $k$ and $r$, while for Tur\'an numbers it has  frequently been considered $n$ large compared with $m$ and $p$, and often focusing on the quantity $\lim_{n\rightarrow\infty} \T(n,m,p)/\binom{n}{p}$ for fixed $m$ and $p$. Thus, for Tur\'an-type problems, the value $\C(n,k,r)$ has usually been studied in the case when $k$ and $r$ are not too far from $n$.

\smallskip

Forge and Ram\'irez Alfons\'in~\cite{For-98} proved that
\begin{equation}\label{upper1}
\CC(n,r)\geq \frac{\binom{n}{r}-1}{r}=:\CC^*_1(n,r). 
\end{equation}
Moreover, Sidorenko~\cite{Sid-82} proved that
$\T(n,r+1,r)\ge \left(\frac{n-r}{n-r+1}\right)\frac{\binom{n}{r}}{r}$. Together with
\eqref{obs:dual1}, we obtain that
\begin{equation}\label{upper2}
\CC(n,r)\geq \C(n,r)=\T(n,n-r,n-r-1)\ge
\left(\frac{r+1}{r+2}\right)\frac{\binom{n}{r+1}}{n-r-1}=:\CC^*_2(n,r).
\end{equation}
Combining~\eqref{upper1} and~\eqref{upper2}, together with a straight forward computation we have
\begin{equation}\label{prop:lb}
\CC(n,r)\geq \max\{\CC^*_1(n,r), \CC^*_2(n,r)\}, 
\end{equation}
where the maximum is attained by the second term if and only if $r\ge \frac{2}{3}(n-1)$.

\smallskip

The following recursive lower bound for covering numbers was obtained by Sch\"onheim~\cite{Sch-64} and, independently, by Katona, Nemetz and Simonovits~\cite{Kat-64} 
\begin{equation}\label{thm:schoenheim1}
 \C(n,r)\geq \left\lceil \frac{n}{r+1}\C(n-1,r-1)\right\rceil
\end{equation}
which can be iterated yielding to
\begin{equation}\label{thm:schoenheim2}
\C(n,r)\geq \left\lceil\frac{n}{r+1}\left\lceil\frac{n-1}{r}\left\lceil\ldots\left\lceil\frac{n-r+1}{2}\right\rceil\ldots\right\rceil\right\rceil\right\rceil=:L(n,r).
\end{equation}

Forge and Ram\'irez Alfons\'in~\cite[Theorem 4.2]{For-98} proved that $\CC(n,r)\le\sum_{i=r+1}^{n-1}\C(i,r-1)$. In this proof, they used the following recursive upper bound that will be useful for us later,

\begin{equation}\label{recursive}
\CC(n,r) \le \CC(n-1,r)+\C(n-1,r-1).
\end{equation}

\section{Results for small and large $r$}\label{ref:small}

In this section, we investigate connected covering numbers for \emph{small }and {\em large} $r$, that is, when $r$ is very close to either $1$ or $n$. Let us start with the following observations.

\begin{remarks}\label{obs:small}
\begin{enumerate}[a)]
\item[]
\item
$\CC(n,0)=1$ since any $1$-element set contains the empty set.
\item
$\CC(n,1)=n-1$ by taking the edges of a spanning tree of $K_n$.
\item
$\CC(n,n-2)=n-1$ by taking all but one $(n-1)$-sets.
\item
$\CC(n,n-1)=1$ by taking the entire set.
\end{enumerate}
\end{remarks}

All these values coincide with the corresponding covering numbers except in the case $r=1$, where $\C(n,1)=\lceil\frac{n}{2}\rceil$.
 
\subsection{Results when $r$ is small}
\ \par
For ordinary covering numbers, Fort and Hedlund~\cite{For-58} have shown that $\C(n,2):=\lceil\frac{n}{3}\lceil\frac{n}{2}\rceil\rceil$ that coincides with the lower bounds given in~\eqref{thm:schoenheim2} when the case $r=2$.

\smallskip

We also have the precise value for the connected case when $r=2$.

\begin{theorem}\label{thm:CCn32}
Let $n$ be a positive integer with $n\ge 3$. Then, we have
$$
\CC(n,2)=\left\lceil \frac{\binom{n}{2}-1}{2}\right\rceil.
$$
\end{theorem}

\begin{proof}
Note that the claimed value coincides with the lower bound $\CC^*_1(n,2)$. This lower bound comes from the fact that every connected covering has a construction sequence, where every new triangle shares at least one edge with an already constructed triangle. We present a construction sequence where indeed every new triangle (except possibly the last one) shares exactly one edge with the already constructed ones.  Therefore, we attain the lower bound.
\begin{figure}[htb] 
\includegraphics[width=.5\textwidth]{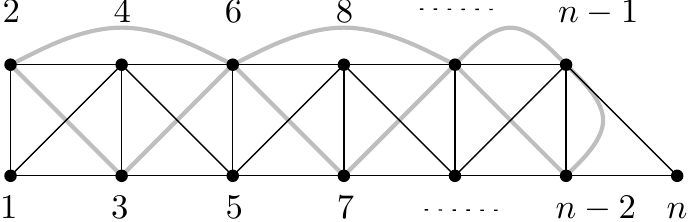}
\caption{Part of the construction proving $\CC(n,2)\leq \CC^*_1(n,2) $.}
\label{fig:CCn32}
\end{figure}
Part of the construction is shown in Figure~\ref{fig:CCn32}. We start presenting the black triangles from left to right. Then we present all triangles of the form $(2i-1, 2i, j)$ for $1\leq i\leq \frac{n}{2}$ and $j\geq 2i+3$. (These are not depicted in the figure.) Now we present the gray triangles from left to right. A gray triangle of the form $(2i, 2i+1, 2i+4)$ is connected to the already presented ones via $(2i-1, 2i, 2i+4)$. Note (as in the figure) the last triangle may indeed share two edges of already presented triangles, depending on the parity of $n$. This accounts for the ceiling in the formula. It is easy to check that all edges are covered.
\end{proof}

\begin{figure}[htb] 
 \includegraphics[width=.7\textwidth]{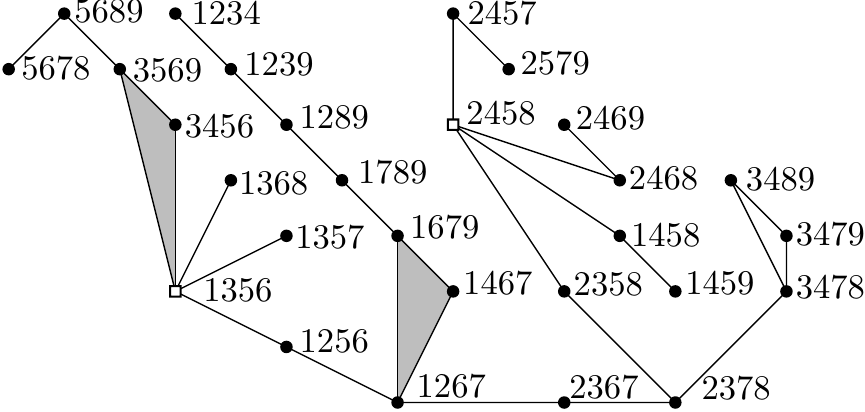}
 \caption{An example proving $\CC(9,3)\leq 28$. The circle-vertices are a covering.}
 \label{fig:CC943}
\end{figure}

The precise value of $C(n,3)$ remains unknown only for finitely many $n$, see~\cite{Mil-74,Mil-81,Ji-08}. The situation for connected coverings is worse.

\begin{theorem}\label{prop:Cn43}
Let $n$ be a positive integer with $4\le n\le 12$. Then, we have
$$
\CC(n,3)=\left\lceil \frac{\binom{n}{3}-1}{3}\right\rceil.
$$
\end{theorem}

\begin{proof}
Note that the claimed value coincides with the lower bound $\CC^*_1(n,3)$. For $n\leq 6$ this is already checked in~\cite{For-98}. Figure~\ref{fig:CC743} proves $\CC(7,3)\leq 12=\CC^*_1(7,3)$.  By using~\eqref{recursive}, $\C(7,2)=7$, and $\CC(7,3)=12$, we obtain that $\CC(8,3)\le 19=\CC^*_1(8,3)$. Figure~\ref{fig:CC943} proves $\CC(9,3)\leq 28=\CC^*_1(9,3)$. From equation~\eqref{recursive} and the fact that $\CC(9,3)=28$ and $\C(9,2)=12$, we conclude that $\CC(10,3)\le 40=\CC^*_1(10,3)$. Now, Figure~\ref{fig:CC1143} proves that $\CC(11,3)\leq 55=\CC^*_1(11,3)$. Finally, to construct a connected covering witnessing $\CC^*_1(12,3)$ we delete the block $\{2,4,6,8\}$ from the covering in Figure~\ref{fig:CC1143}. One can check that this still leaves a covering $\mathcal{B}$, whose graph now has three components. Now, we take the following (disconnected) $(11,3,2)$-covering:
$$
\left\{\begin{array}{l}
\{1,3,11\} , \{1,4,6\} , \{1,2,8\} , \{1,5,9\} , \{1,7,10\} , \{3,4,9\} , \{2,3,10\} , \{3,5,6\} , \\ \{3,7,8\} ,  \{2,4,6\} , \{4,5,7\} , \{4,10,11\} , \{4,6,8\} , \{2,5,11\} , \{2,7,9\} , \{5,8,10\} , \\ \{6,7,11\} , \{8,9,11\} , \{6,9,10\}
\end{array}\right\}.
$$
We add to each of these block the element $12$ and thus together with $\mathcal{B}$ obtain a $(12,4,3)$-covering $\mathcal{B}'$. To see that $\mathcal{B}'$ is connected, note that each of the blocks containing $12$ is connected to a block from $\mathcal{B}$. Moreover, the blocks $\{1,  4,  6 , 12\}, \{2,  4,  6, 12 \},\{4,  6, 8, 12 \}$ form a triangle and each of them has a neighbor in a different component of $G(\mathcal{B})$. Thus, $G(\mathcal{B}')$ is connected and $\mathcal{B}'$ has $73$ blocks which coincides with $\CC^*_1(12,3)$.
\end{proof}

\begin{figure}[htb] 
 \includegraphics[width=.8\textwidth]{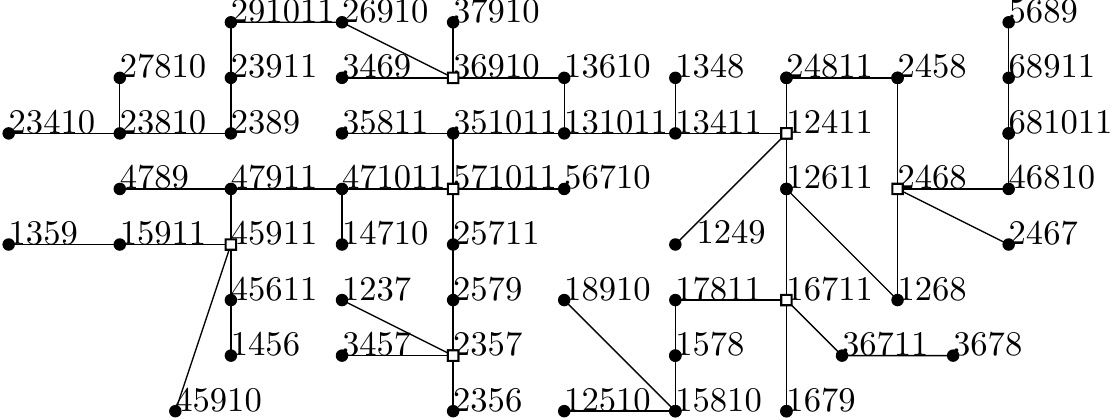}
 \caption{An example proving $\CC(11,3)\leq 55$. The circle-vertices are a covering.}
 \label{fig:CC1143}
\end{figure}

Theorem~\ref{prop:Cn43} supports the following

\begin{conjecture}
For every positive integer $n\ge 4$, we have
$$
\CC(n,3)=\CC^*_1(n,3).
$$
\end{conjecture}

Even, more ambitious,

\begin{question}
Let $n$ and $r$ be two positive integers such that $n\ge r+1\ge 4$. Is it true that if $\CC(n,r)=\CC^*_1(n,r)$ then $\CC(n',r)=\CC^*_1(n',r)$
for every integer $n'\geq n$ ?
\end{question}

\subsection{Results when $r$ is large}

\begin{theorem}\label{thm:Turan1}
Let $n$ be a positive integer with $n\ge 3$. Then, we have
$$
\CC(n,n-3)=\binom{\left\lceil\frac{n}{2}\right\rceil}{2}+\binom{\left\lfloor\frac{n}{2}\right\rfloor}{2}+1.
$$
\end{theorem}

\begin{proof}
The parameter $\C(n,n-3)=\T(n,3,2)$ was determined already by Mantel in 1907~\cite{Man-07} and is $\binom{\lceil\frac{n}{2}\rceil}{2}+\binom{\lfloor\frac{n}{2}\rfloor}{2}$. Tur\'an proved that the \emph{unique} minimal configuration of sets of size $2$ hitting all sets of size $3$ of an $n$-set are the edges of two vertex-disjoint complete graphs $K_{\lceil\frac{n}{2}\rceil}$ and $K_{\lfloor\frac{n}{2}\rfloor}$, see~\cite{Tur-41}.

\smallskip

Now, by~\eqref{obs:dual1} and~\eqref{obs:dual2}, the covering corresponding to the Tur\'an-system is connected if and only if the graph whose edges correspond to the blocks of the Tur\'an-system is connected. Thus, since the unique optimal construction by Tur\'an is not connected but can be made connected by adding a single edge connecting the two complete graphs, this is optimal with respect to connectivity. Therefore, $\CC(n,n-3)=\T(n,3,2)+1$, giving the result.
\end{proof}

\begin{proposition}\label{prop:Turan2}
Let $n\neq 5,6,8,9$ be a positive integer with $n\ge 4$. Then, we have
$$
\CC(n,n-4)\leq \begin{cases}
 m(m-1)(2m-1) & \text{if }n=3m, \\
 m^2(2m-1) & \text{if }n=3m+1,\\
 m^2(2m+1) & \text{if }n=3m+2.\\
\end{cases}
$$
If $n=5,6,9$  the value of $\CC(n,n-4)$ is one larger than claimed in the formula. 
Further, $\CC(8,4)\in\{20,21\}$, i.e., it remains open if the above formula has to be increased by one or not in order to give the precise value.
\end{proposition}

\begin{proof}
We will show that a Tur\'an-system $\mathcal{D}$ verifying the claimed bounds due to Kostochka~\cite{Kos-82} is connected. Indeed the construction of~\cite{Kos-82} is a parametrized family of Tur\'an-systems, each of whose members attains the claimed bound. Our construction results from picking special parameters:

Assume that $n\geq 12$ and $n$ is divisible by $3$. Split $[n]$ into three sets $A_1, A_2, A_3$ of equal size. Pick special elements $x_i, y_i\in A_i$ and denote $B_i:=A_i\setminus\{x_i,y_i\}$ for $i=0,1,2$. The blocks of $\mathcal{D}$ consist of $3$-element sets $\{a,b,c\}$ of the following forms:
\begin{itemize}
 \item[$L_i$:] $a,b,c \in A_i$,
 \item[$T1_i$:] $a=x_i$ and $b,c\in A_{i+1}$,
 \item[$T2_i$:] $a=y_i$ and $b,c\in B_{i-1}\cup\{x_{i+1},y_{i+1}\}$,
 \item[$T3_i$:] $a\in B_i$ and $b,c\in B_{i-1}\cup\{x_{i+1},y_{i-1}\}$
\end{itemize}
where $i=0,1,2$, and addition of indices is understood modulo $3$. 

Let us now show that $\mathcal{D}$ is connected. Clearly, all blocks in a given $A_i$ are connected and all $2$-element subsets in each $A_i$ are covered by a block in this $A_i$. Thus, it suffices to verify that there are two $2$-element sets $\{e,f\}\subseteq A_0$ and $\{e',f'\}\subseteq A_2$ which can be connected by a sequence of blocks of $\mathcal{D}$, because then any block in $A_0$ containing $\{e,f\}$ is connected to any block in $A_2$ containing $\{e',f'\}$. The connectivity of $\mathcal{D}$ then follows by the symmetry of the construction. Let $\{e,f\}\subseteq B_0$. Take $\{e,f,y_1\}\in T2_1$, then $\{e,y_1,y_2\}\in T2_1$, and then $\{e,f',y_2\}\in T3_0$, where $f'\in B_2$, i.e,  $\{y_2,f'\}\subseteq B_2$.

Now, following~\cite{Kos-82} deleting any element of such a system yields a Tur\'an-system $\mathcal{D}'$ of the claimed size for $n'=n-1$. We can just delete any $x_i$, since these are not used for connectivity.
Following~\cite{Kos-82}, two elements can be removed from $\mathcal{D}$ to obtain a Tur\'an-system $\mathcal{D}''$ of the claimed size for $n''=n-2$, if the set formed by these two elements belongs to exactly $\frac{n}{3}-1$ blocks. This is the case for $\{x_i, x_{i+1}\}$, which belongs to exactly $\frac{n}{3}-1$ blocks from $T1_i$. Again, this preserves connectivity.

We are left with the cases $n\leq 9$. In~\cite{Sid-82} it is shown that the Tur\'an-systems of the claimed size for $n=9$ are exactly the members of the family constructed in~\cite{Kos-82}. There are exactly two such systems:

In both cases $[9]$ is split into three sets $A_1, A_2, A_3$ of size $3$. In the first system we pick a $x_i\in A_i$ and denote $A_i\setminus \{x_i\}$ by $B_i$. The blocks then are the $3$-element sets $\{a,b,c\}$ of the following forms:
\begin{itemize}
 \item[$L_i$:] $a,b,c \in A_i$,
 \item[$T1_i$:] $a=x_i$ and $b,c\in A_{i+1}$,
 \item[$T2_i$:] $a\in B_i$ and $b,c\in B_{i-1}\cup\{x_{i+1}\}$.
\end{itemize}

The second system coincides with an instance of a construction due to Tur\'an~\cite{Tur-61}. It consists of the following $3$-element sets:
\begin{itemize}
 \item[$L_i$:] $a,b,c \in A_i$,
 \item[$T1_i$:] $a\in A_i$ and $b,c\in A_{i+1}$.
\end{itemize}

It is easy to check that both systems are not connected. On the other hand, the second one can be made connected adding a single block taking one element from each $A_i$. This proves the claim for $n=9$. Further, removing any vertex not contained in the added block, one obtains a connected Tur\'an-system for $n=8$ with $21$ blocks. While there are Tur\'an-systems showing $\T(8,4,3)=20$ we do not know if there is any such connected system.

See Figure~\ref{fig:CC743} for proving our statement for $n=7$, Theorem~\ref{thm:CCn32} for $n=6$, and Remark~\ref{obs:small} for $n=4,5$.
\end{proof}

A famous conjecture of Tur\'an \cite{Tur-61} states that the bounds in Proposition~\ref{prop:Turan2} are best possible for $\C(n,n-4)$. By combining~\eqref{l1} and Proposition~\ref{prop:Turan2}, for $n\geq 10$ we have
\begin{equation}\label{contur}
\C(n,n-4)\leq\CC(n,n-4)\leq \begin{cases}
 m(m-1)(2m-1) & \text{if }n=3m, \\
 m^2(2m-1) & \text{if }n=3m+1,\\
 m^2(2m+1) & \text{if }n=3m+2.\\
\end{cases}
\end{equation}
Tur\'an's conjecture has been verified for all $n\le13$ by~\cite{Sid-82} and so, by~\eqref{contur}, the connected covering number can also be determined for these same values.

\smallskip

Towards proving Tur\'an's conjecture, it would be of interest to investigate the following.

\begin{question} Is it true that one of the inequalities in~\eqref{contur} is actually an equality ?
\end{question}

Bounds and precise values for all $\CC(n,r)$ with $n\leq 14$ are given in Table~\ref{current}. All the exact values previously given in~\cite{For-98} for the same range  have been improved by using our above results. 

\begin{center}
\begin{table}[h] \label{table}
\centering
\footnotesize
\noindent\makebox[\textwidth]{%
\begin{tabular}{|c|c|c|c|c|c|c|c|c|c|c|c|c|c|c|}

 \hline
$r\setminus n$   &$1$ &$2$    &$3$       &$4$ & $5$ &$6$ &$7$ &
$8$ &$9$ & $10$& $11$ &$12$ &$13$&$14$  \\

\hline
  $0$          & $1$ &$1$ &  $1$&$1$ & $1$& $1$& $1$& $1$& $1$&$1$ & $1$&$1$ &$1$ & $1$  \\
\hline
  $1$          &  &$1$ & $2$ &$3$ &$4$ &$5$&$6$ &$7$ &$8$ &$9$ &$10$ &$11$ &$12$ & $13$  \\
             
  \hline
     $2$          &  & &$1$ &$3$ & $5^{e,t}$& $7^{e}$&$10^e$ &$14^e$ & $18^e$& $22^e$&$27^e$ &$33^e$& $39^e$&$45^e$  \\
             
  \hline
     $3$         & &  & & $1$ &$4$&$7^{p,t}$ &$12^{p,u}$& $19^p$&$28^p$ &$40^p$ &$55^p$ & $73^p$& $[95^l,97^r]$& $[121^l,123^r]$   \\
             
  \hline
     $4$         & &  & &  & $1$& $5$&$10^t$ &$[20,21^u]$ &$[32^l,35^r]$ & $[53^l,59^r]$&$[83^l,89^r]$ &$[124^l,136^r]$ & $[179^l,193^r]$& $[250^l,271^r]$   \\
             
  \hline
     $5$          & & & &  & &$1$ &$6$ &$13^t$ &$31^u$ &$[51^l,61^r]$ &$[96^a,111^r]$ &$[159^l,177^r]$  & $[258^l,290^r]$&  $[401^l,447^r]$  \\
             
  \hline
     $6$         & &  & & &  &  & $1$& $7$&$17^t$ & $45^u$& $[84^a,95^r]$&$[165^a,195^r]$ & $[286^l,327^r]$& $[501^l,572^r]$   \\
             
  \hline
     $7$         &  & & & &  & & &$1$ &$8$&$21^t$ &$63^u$ &$[126^a,147^r]$ &$[269^a,323^r]$ &  $[491^l,587^r]$  \\
             
  \hline
     $8$         & &  & & &  & & & &$1$ & $9$& $26^t$&$84^u$  &$[185^a,210^r]$ &  $[419^a,505^r]$  \\
             
  \hline
     $9$          &&  & & &  & & & & &$1$ & $10$&$31^t$ &$112^u$ & $[259^s,297^r]$  \\
             
  \hline
     $10$         & & & & & & & & & & &$1$ &$11$&$37^t$ & $[143^s,144^u]$   \\
             
  \hline
     $11$         & &  & & & & & & & & & &$1$ &$12$& $43^t$  \\
             
  \hline
     $12$       &  &   & &  & &  &  & & && & &$1$  & $13$  \\
             
  \hline
   $13$         & &  &  & & & & & & & & &  & & $1$ \\
             
     \hline          

\end{tabular}}
\vspace{.2cm}
\caption{\label{current}Bounds and values of $\CC(n,r)$ for $n\le 14$.}
\normalsize \underline{ Key of Table~\ref{current} : }\\ \ \\
\begin{itemize}
 \item[$r$]--- Upper bound for $\CC(n,r)$(from {Equation}~\eqref{recursive})
\item[$e$]--- Exact values for $\CC(n,2)$ (Theorem~\ref{thm:CCn32}) 
\item[$t$]--- Exact values for $\CC(n,n-3)$ (Theorem~\ref{thm:Turan1})
 \item[$l$]--- Lower bound $\CC^*_1(n,r)$
\item[$p$]--- Some exact values for $\CC(n,3)$ (Theorem~\ref{prop:Cn43}) 
\item[$u$]--- Upper bound for $\CC(n,n-4)$ (Proposition~\ref{prop:Turan2})
  \item[$s$]--- Lower bound for $\C(n,r)$ (from {Equation}~\eqref{thm:schoenheim1})
   \item[$a$]--- Lower bounds for $\C(n,r)$ (from \cite{App-03})
\end{itemize}
\end{table}
\end{center}

Table~\ref{current} led us to consider the following.

\begin{question}
Is the sequence $\left(\CC(n,i)\right)_{0\leq i\leq n-1}$ {\em unimodal} for every $n$ ? or perhaps  {\em logarithmically concave}\footnote{A finite sequence of real numbers $\{a_1,a_2,\dots ,a_n\}$ is said to be {\em unimodal} (resp. {\em logarithmically concave} or {\em log-concave}) if there exists a $t$ such that $s_1\le s_2\le\cdots \le s_t$ and $s_t\ge s_{t+1}\ge\cdots \ge s_n$ (resp. if $a_i^2\ge a_{i-1}a_{i+1}$ holds for every $a_i$ with $1\le i \le n-1$). Notice that a log-concave sequence is unimodal.} ? 
\end{question}

\section{A general upper bound}\label{sec:gen}

Let $n$ and $r$ be positive integers such that $n\ge r+1 \ge 3$. Forge and Ram\'irez Alfons\'in  \cite{For-98} obtained the following general upper bound
\begin{equation}\label{defs}
\Su(n,r) := \sum_{i=1}^{\left\lfloor\frac{n-r+1}{2}\right\rfloor}\binom{n-2i}{r-1} + \left\lfloor\frac{n-r}{2}\right\rfloor \ge \CC(n,r).
\end{equation}
Let us notice that the upper bounds obtained by applying the recursive equation ~\eqref{recursive}, that were used in Table~\ref{current}, are better than the one given by~\eqref{defs}. Moreover, by iterating ~\eqref{recursive} it can be obtained 
\begin{equation}\label{defs1}
\CC(n,r) \le \sum_{i=r}^{n-1} \C(i,r-1).
\end{equation}
Although ~\eqref{defs1} might be used to get an {\em explicit} upper bound for $s(n,r)$, it is not clear how good it would be since that would depend on the known exact values and the upper bounds of $\C(n,r)$ used in the recurrence (and thus intrinsically difficult to compute). On the contrary, in \cite{For-98} Equation \eqref{defs} was used to give the best known (to our knowledge) explicit upper bound for $s(n,r)$.

\smallskip

In this section, we will construct a connected $(n,r+1,r)$-covering giving an upper bound for $\CC(n,r)$ better than  $\Su(n,r)$ and so, yielding a better upper bound for $s(n,r)$ than that given in \cite{For-98}.

\begin{theorem}\label{construction}
Let $n$ and $r$  be positive integers such that $n\ge r+1\ge 3$. 
Then
$
\CC(n,r) \le \N(n,r),
$
where
\begin{equation}\label{defn}
\N(n,r) := \sum_{i=0}^{\left\lceil\frac{n-r}{2}\right\rceil-1}(n-r-2i)\binom{r-2+2i}{r-2} + \left\lceil\frac{n-r}{2}\right\rceil-1 + \delta_0\C(n-2,r-2),
\end{equation}
and $\delta_0$ is the parity function of $n-r$, that is,
$
\delta_0 = \left\{\begin{array}{ll}
 0 & \text{ if } n-r \text{ is odd},\\
 1 & \text{otherwise}.\\
\end{array}\right.
$
\end{theorem}

\begin{proof} 
From this point on, for any positive integer $s$, we will denote $[s]:=\{1,\ldots,s\}$  and
by ${[s] \choose t}$ the set of all $t$-subsets of $[s]$. Moreover,
 for any subset of integers $\{b_1,\ldots,b_s\}$, we may suppose
that $b_i<b_j$ for all integers $i$ and $j$ such that $1\le i<j\le s$.

\begin{case}Suppose that $n-r$ is odd and let $m$ such that
$n-r=2m+1$. We will construct a connected $(r+2m+1,r+1,r)$-covering
of size
$$
m+ \sum_{i=0}^{m}{r-2+2i\choose r-2}(2m+1-2i).
$$
We consider a particular $(r+2m+1,r+1,r)$-covering, which is
constituted by a large number of blocks but whose associated graph
has a small number of connected components. For any
$i\in\{0,\ldots,m\}$, let $\mathcal{N}_i$ be the following subset of
$(r+1)$-subsets of $[r+2m+1]$:
$$
\mathcal{N}_i := \left\{ \{b_1,\ldots,b_{r+1}\}\ \middle|\
\begin{array}{l}
\{b_1,\ldots,b_{r-2}\}\in {[r+2i-2]\choose r-2} \\
b_{r-1}=r+2i-1,\ b_r=r+2i \\
b_{r+1}\in\{r+2i+1,\ldots,r+2m+1\}\end{array}\right\}.
$$

\begin{claim}\label{claim1} The set $\displaystyle\bigcup_{i=0}^{m}\mathcal{N}_i$
is a $(r+2m+1,r+1,r)$-covering.\end{claim}

Let $b=\{b_1,\ldots,b_r\}\in {[r+2m+1]\choose r}$.  If
$b_{r-1}=r-1+2i$ for some $i\in\{0,\ldots,m\}$, then $b\subset B$ for some
$B\in \mathcal{N}_i$. The same occurs if $b_{r-1}=r+2i$.

\begin{claim}\label{claim2}  The graph $G\left(\mathcal{N}_i\right)$ is connected,
for any $i\in\{0,\ldots,m\}$.\end{claim}

Let $B=\{b_1,\ldots,b_{r+1}\}$ and $C=\{c_1,\ldots,c_{r+1}\}$ in
$\mathcal{N}_i$.  Clearly, $B$ is adjacent to
$\{b_1,\ldots,b_{r},c_{r+1}\}$ in $G(\mathcal{N}_i)$. Since
$b_{r-1}=c_{r-1}$, $b_r=c_r$ and
$\{d_1,\ldots,d_{r-2},c_{r-1},c_r,c_{r+1}\}\in \mathcal{N}_i$ for
all $\{d_1,\ldots,d_{r-2}\}\subset[r-2+2i]$, then there exists a
path from $B$ to $C$.

\begin{claim}\label{claim3}  There exists a $(r+1)$-subset $C_i$ such that
$G(\mathcal{N}_i\cup C_i\cup \mathcal{N}_{i+1})$ is connected for
any $i\in\{0,\ldots,m-1\}$.\end{claim}

Let $B_i=\{1,\ldots,r-2,r-1+2i,r+2i,r+1+2i\}\in \mathcal{N}_i$ and
$B_{i+1}=\{1,\ldots,r-2,r+1+2i,r+2+2i,r+3+2i\}\in
\mathcal{N}_{i+1}$. Then, the $(r+1)$-subset
$C_i=\{1,\ldots,r-2,r+2i,r+1+2i,r+2+2i\}$ is adjacent to $B_i$ and
$B_{i+1}$ in $G(\mathcal{N}_i\cup C_i\cup \mathcal{N}_{i+1})$. This
concludes the proof of Claim 3.

By Claims \ref{claim1}, \ref{claim2} and \ref{claim3}, we obtain that $
\left(\bigcup_{i=0}^{m}\mathcal{N}_i\right)\bigcup\left(\bigcup_{i=0}^{m-1}C_i\right)
$ is a connected $(r+2m+1,r+1,r)$-covering. Finally, since
$|\mathcal{N}_i|={r-2+2i\choose r-2}(2m+1-2i)$ for any
$i\in\{0,\ldots,m\}$, the theorem holds in this case. \end{case}

\begin{case}
Suppose  $n-r$ is even and let $m$ be such that
$n-r=2m$. We are going to construct a $(r+2m,r+1,r)$-connected
covering of size
$$m-1+\C(r-2+2m,r-2)+\sum_{i=0}^{m-1}{r-2+2i\choose r-2}(2m-2i).$$

As already defined in Case 1, we consider the collection
$\mathcal{N}_i$ of $(r+1)$-subsets of $[r+2m]$ defined by $
\mathcal{N}_i := {[r+2i-2]\choose
r-2}\times\{r+2i-1\}\times\{r+2i\}\times\{r+2i+1,\ldots,r+2m\}$ for
any $i\in\{0,\ldots,m-1\}$.
Let $\mathcal{C}$ be a $(r+2m-2,r-1,r-2)$-covering of size
$\C(r+2m-2,r-2)$ and consider the set $ \mathcal{N}_m := \left\{
B\cup\{r+2m-1,r+2m\}\ \middle|\ B\in\mathcal{C}\right\}$. Then, one
can check that $\bigcup\limits_{i=0}^{m}\mathcal{N}_i$ is a
$(r+2m,r+1,r)$-covering. Similarly as in  the proofs of Claims \ref{claim2} and
\ref{claim3}, it follows that $G(\mathcal{N}_i)$ is connected for any
$i\in\{0,\ldots,m-1\}$ and there exists a $(r+1)$-subset $C_i$ such that
$G(\mathcal{N}_i\cup C_i\cup \mathcal{N}_{i+1})$ is connected for
any $i\in\{0,\ldots,m-2\}$.

\begin{claim}\label{claim4}. For any $B\in \mathcal{N}_m$, there exist
$i\in\{0,\ldots,m-1\}$ and $C\in \mathcal{N}_{i}$ such that $B$ is
adjacent to $C$ in the graph $G(\mathcal{N}_i\cup
\mathcal{N}_{m})$.\end{claim}

Let $B=\{b_1,\ldots,b_{r-1},r+2m-1,r+2m\}\in \mathcal{N}_m$. If
$b_{r-1}=r+2i-1$ for some $i\in\{0,\ldots,m-1\}$, then
$\{b_1,\ldots,b_{r-2}\}\in {[r+2i-2]\choose r-2}$. Let $ C
=\{b_1,\ldots,b_{r-2},r+2i-1,r+2i,r+2m\}$, by definition
$C\in\mathcal{N}_i$ and moreover, since
$\{b_1,\ldots,b_{r-2},r+2i-1,r+2m\}\subset B$ and
$\{b_1,\ldots,b_{r-2},r+2i-1,r+2m\}\subset C$, we deduce that $B$
and $C$ are adjacent in the graph $G(\mathcal{N}_i\cup
\mathcal{N}_{m})$. Either, if $b_{r-1}=r+2i$ for some
$i\in\{0,\ldots,m-1\}$, we have that $\{b_1,\ldots,b_{r-2}\}\in
{[r+2i-1]\choose r-2}$. We distinguish two cases on the value of
$b_{r-2}$. First, if $b_{r-2}<r+2i-1$, then
$\{b_1,\ldots,b_{r-2}\}\in {[r+2i-2]\choose r-2}$. Consider now
$C=\{b_1,\ldots,b_{r-2},r+2i-1,r+2i,r+2m\}$. As above, since
$\{b_1,\ldots,b_{r-2},r+2i,r+2m\}\subset B$ and
$\{b_1,\ldots,b_{r-2},r+2i,r+2m\}\subset C$, we deduce that $B$ and
$C$ are adjacent in the graph $G(\mathcal{N}_i\cup
\mathcal{N}_{m})$. Finally, suppose that $b_{r-2}=r+2i-1$ and let
$\alpha\in[r+2i-2]\setminus\{b_1,\ldots,b_{r-3}\}$ and  $
C=\{b_1,\ldots,b_{r-3},r+2i-1,r+2i,r+2m\}\cup\{\alpha\} \in {[r+2m]
\choose r+1}$. Since $\{b_1,\ldots,b_{r-3},r+2i-1,r+2i,r+2m\}\subset
B$ and $\{b_1,\ldots,b_{r-3},r+2i-1,r+2i,r+2m\}\subset C$, we deduce
that $B$ and $C$ are adjacent in the graph $G(\mathcal{N}_i\cup
\mathcal{N}_{m})$. This concludes the proof of Claim \ref{claim4}.
\end{case}

Hence,
$(\bigcup_{i=0}^{m}\mathcal{N}_i)\cup(\bigcup_{i=0}^{m-2}C_i)$ is
a connected $(r+2m,r+1,r)$-covering. Since
$|\mathcal{N}_i|={r-2+2i\choose r-2}(2m-2i)$ for any
$i\in\{0,\ldots,m-1\}$ and $|\mathcal{N}_m|=\C(n-2,r-2)$, the theorem
holds.
\end{proof}

Let us illustrate the construction given in the above theorem.
\begin{example} $\N(7,4)=10$. 
We consider  
$$\mathcal{N}_0=\{12345,12346,12347\} \text{ and } \mathcal{N}_1=\{12567,13567,14567,23567,24567,34567\}.$$ 
It can be checked that $\mathcal{N}_0\cup \mathcal{N}_1$ is  a $(7,5,4)$-covering and $G(\mathcal{N}_0)$ and $G(\mathcal{N}_1)$ are connected. Now, by taking $C_0=12456$, it follows that  $G(\mathcal{N}_0\cup C_0\cup \mathcal{N}_{1})$ is  connected.
\end{example}

We may now show that $\Su(n,r)> \N(n,r)$. For this we need first the following Theorem and Proposition.

\begin{theorem}\label{thm1}
Let $r$ and $n$ be positive integers such that $n\ge r+1 \ge 3$. Then,
$$
\Su(n,r) = \N(n,r) + \sum_{i=0}^{\left\lfloor\frac{n-r}{2}\right\rfloor-1}\left(\left\lfloor\frac{n-r}{2}\right\rfloor-i\right)\binom{r-2+2i}{r-3} + \delta_0\left(1-\C(n-2,r-2)\right),
$$
where $\delta_0$ is the parity function of $n-r$.
\end{theorem}

\begin{proof}
By induction on $n>r$. From~\eqref{defs} and~\eqref{defn},  the identity is verified for $n=r+1$ and $n=r+2$. Suppose now that the identity is verified for a certain value of $n$ and let $D$ be the difference
$$
D := \left(\Su(n+2,r)-\N(n+2,r)\right)-\left(\Su(n,r)-\N(n,r)\right).
$$
Then 
$$
\Su(n+2,r) = \N(n+2,r) + \left(\Su(n,r)-\N(n,r)\right) + D.
$$
By using~\eqref{defs}, we obtain
$$
\Su(n+2,r) - \Su(n,r)
\begin{array}[t]{l}
 = \displaystyle\sum_{i=1}^{\left\lfloor\frac{n-r+1}{2}\right\rfloor+1}\binom{n+2-2i}{r-1} + \left\lfloor\frac{n-r}{2}\right\rfloor +1 - \sum_{i=1}^{\left\lfloor\frac{n-r+1}{2}\right\rfloor}\binom{n-2i}{r-1} - \left\lfloor\frac{n-r}{2}\right\rfloor \\ \ \\
 = \displaystyle\binom{n}{r-1} + 1.
\end{array}
$$
By using~\eqref{defn}, we have
$$
\N(n+2,r) - \N(n,r)
\begin{array}[t]{l}
 = \begin{array}[t]{l}
 \displaystyle\sum_{i=0}^{\left\lceil\frac{n-r}{2}\right\rceil}(n+2-r-2i)\binom{r-2+2i}{r-2} + \left\lceil\frac{n-r}{2}\right\rceil+ \delta_0\C(n,r-2) \\ \ \\
 -\displaystyle\sum_{i=0}^{\left\lceil\frac{n-r}{2}\right\rceil-1}(n-r-2i)\binom{r-2+2i}{r-2} - \left\lceil\frac{n-r}{2}\right\rceil +1 - \delta_0\C(n-2,r-2) \\ \end{array} \\ \ \\ 
 = \begin{array}[t]{l}
 \displaystyle\sum_{i=0}^{\left\lceil\frac{n-r}{2}\right\rceil-1}2\binom{r-2+2i}{r-2} + \left(n+2-r-2\left\lceil\frac{n-r}{2}\right\rceil\right)\binom{r-2+2\left\lceil\frac{n-r}{2}\right\rceil}{r-2} \\ \ \\
 +\delta_0\left(\C(n,r-2)-\C(n-2,r-2)\right) +1. \\
 \end{array}
\end{array}
$$
Moreover, 
for $n-r$ odd, it follows that
$$
\N(n+2,r) - \N(n,r) = \begin{array}[t]{l} \displaystyle\sum_{i=0}^{\left\lceil\frac{n-r}{2}\right\rceil}2\binom{r-2+2i}{r-2} + (\delta_0-1)\binom{n-1}{r-2} \\ \ \\
 + \delta_0\left(\C(n,r-2)-\C(n-2,r-2)\right) +1. \end{array}
$$
Therefore
$$
D = \binom{n}{r-1} -  \sum_{i=0}^{\left\lceil\frac{n-r}{2}\right\rceil}2\binom{r-2+2i}{r-2} + (1-\delta_0)\binom{n-1}{r-2} + \delta_0\left(\C(n-2,r-2)-\C(n,r-2)\right).
$$
From the identity
$
\binom{r-2+2i}{r-2} = \binom{r-1+2i}{r-2} - \binom{r-2+2i}{r-3},
$
we obtain 
that
$$
\sum_{i=0}^{\left\lceil\frac{n-r}{2}\right\rceil}2\binom{r-2+2i}{r-2}
\begin{array}[t]{l}
 = \displaystyle\sum_{i=0}^{\left\lceil\frac{n-r}{2}\right\rceil}\binom{r-2+2i}{r-2} + \sum_{i=0}^{\left\lceil\frac{n-r}{2}\right\rceil}\binom{r-1+2i}{r-2} - \sum_{i=0}^{\left\lceil\frac{n-r}{2}\right\rceil}\binom{r-2+2i}{r-3} \\ \ \\
 = \displaystyle\sum_{i=r-2}^{r-1+2\left\lceil\frac{n-r}{2}\right\rceil}\binom{i}{r-2} - \sum_{i=0}^{\left\lceil\frac{n-r}{2}\right\rceil}\binom{r-2+2i}{r-3} \\ \ \\
 = \displaystyle\binom{r+2\left\lceil\frac{n-r}{2}\right\rceil}{r-1} - \sum_{i=0}^{\left\lceil\frac{n-r}{2}\right\rceil}\binom{r-2+2i}{r-3}.
\end{array}
$$
Thus,
$$
D =
\begin{array}[t]{l}
\displaystyle\binom{n}{r-1} - \binom{r+2\left\lceil\frac{n-r}{2}\right\rceil}{r-1} + (1-\delta_0)\binom{n-1}{r-2} + \sum_{i=0}^{\left\lceil\frac{n-r}{2}\right\rceil}\binom{r-2+2i}{r-3} \\ \ \\
 + \delta_0\left(\C(n-2,r-2)-\C(n,r-2)\right).
\end{array}
$$
If $n-r$ is even, then $\delta_0=1$ and
$$
\binom{n}{r-1} - \binom{r+2\left\lceil\frac{n-r}{2}\right\rceil}{r-1} + (1-\delta_0)\binom{n-1}{r-2} = \binom{n}{r-1} - \binom{n}{r-1} = 0.
$$
Either, if $n-r$ is odd, then $\delta_0=0$ and
$$
\binom{n}{r-1} - \binom{r+2\left\lceil\frac{n-r}{2}\right\rceil}{r-1} + (1-\delta_0)\binom{n-1}{r-2}
\begin{array}[t]{l}
 = \displaystyle\binom{n}{r-1} - \binom{n+1}{r-1} + \binom{n-1}{r-2} \\ \ \\
 = -\displaystyle\binom{n}{r-2}+\binom{n-1}{r-2} \\ \ \\
 = -\displaystyle\binom{n-1}{r-3}.
\end{array}
$$
It follows that
$$
 D
\begin{array}[t]{l}
 = \displaystyle\sum_{i=0}^{\left\lceil\frac{n-r}{2}\right\rceil}\binom{r-2+2i}{r-3} + (\delta_0-1)\binom{n-1}{r-3} + \delta_0\left(\C(n-2,r-2)-\C(n,r-2)\right) \\ \ \\
 = \displaystyle\sum_{i=0}^{\left\lfloor\frac{n-r}{2}\right\rfloor}\binom{r-2+2i}{r-3} + \delta_0\left(\C(n-2,r-2)-\C(n,r-2)\right).
\end{array}
$$
Now, with the induction hypothesis, we obtain
$$
\Su(n+2,r) - \N(n+2,r)
\begin{array}[t]{l}
 = \displaystyle\left(\Su(n,r)-\N(n,r)\right) + D \\ \ \\
 = \begin{array}[t]{l} \displaystyle\sum_{i=0}^{\left\lfloor\frac{n-r}{2}\right\rfloor-1}\left(\left\lfloor\frac{n-r}{2}\right\rfloor-i\right)\binom{r-2+2i}{r-3} + \delta_0\left(1-\C(n-2,r-2)\right) \\ \ \\
 + \displaystyle\sum_{i=0}^{\left\lfloor\frac{n-r}{2}\right\rfloor}\binom{r-2+2i}{r-3} + \delta_0\left(\C(n-2,r-2)-\C(n,r-2)\right) \end{array} \\ \ \\
 = \displaystyle\sum_{i=0}^{\left\lfloor\frac{n-r}{2}\right\rfloor}\left(\left\lfloor\frac{n-r}{2}\right\rfloor+1-i\right)\binom{r-2+2i}{r-3} + \delta_0\left(1-\C(n,r-2)\right).
\end{array}
$$
\end{proof}

\begin{theorem}\label{th:gen}
Let $r$ and $n$ be positive integers such that $n-r$ is an even number. Then,
$$
\Su(n,r) \ge \N(n,r) + \sum_{i=0}^{\frac{n-r}{2}-2}\left(\frac{n-r}{2}-i-1\right)\binom{r-2+2i}{r-3}.
$$
\end{theorem}

\begin{proof}
It is known \cite[page 7]{Gor-95} that $
\C(n,r) \le \binom{n-2}{r-1} + \C(n-2,r).
$
By applying this inequality  repeatedly we have
$$
\C(n-2,r-2) \le \sum_{i=0}^{\frac{n-r}{2}-1}\binom{r-2+2i}{r-3} + 1.
$$
Then, we deduce from Theorem~\ref{thm1} that
$$
\Su(n,r)
\begin{array}[t]{l}
 = \N(n,r) + \displaystyle\sum_{i=0}^{\frac{n-r}{2}-1}\left(\frac{n-r}{2}-i\right)\binom{r-2+2i}{r-3} + 1-\C(n-2,r-2) \\ \ \\
 \ge \N(n,r) + \displaystyle\sum_{i=0}^{\frac{n-r}{2}-1}\left(\frac{n-r}{2}-i\right)\binom{r-2+2i}{r-3} - \sum_{i=0}^{\frac{n-r}{2}-1}\binom{r-2+2i}{r-3} \\ \ \\
 = \N(n,r) + \displaystyle\sum_{i=0}^{\frac{n-r}{2}-2}\left(\frac{n-r}{2}-i-1\right)\binom{r-2+2i}{r-3}.
\end{array}
$$
\end{proof}

 \section{Asymptotics}\label{sec:asym}
 In~\cite{Rod-85} R\"odl uses the probabilistic method to show the existence of \emph{asymptotically good coverings}. Restricted to our case this means that
 $$\frac{\C(n,r)}{\binom{n}{r}} \to \frac{1}{r+1}\text{ as }n\to \infty.$$
 
 Since $\CC(n,r)\leq 2\C(n,r)$ (see~\cite{For-98}) we immediately obtain:
 $$\frac{\CC(n,r)}{\binom{n}{r}} \to a\leq \frac{2}{r+1}\text{ as }n\to \infty.$$

In~\cite{For-98} it was shown that 
 $$\frac{\Su(n,r)}{\binom{n}{r}} \to \frac{1}{2}\text{ as }n\to \infty$$
 and since by Theorem~\ref{th:gen} the difference $\N(n,r)-\Su(n,r)$ is in $\mathcal{O}(n^{r-1})$ we have the same asymptotic behavior for $\N(n,r)$.
 
It is however still a topic of research to find explicit constructions witnessing the bound of R\"odl, see~\cite{Kuz-00}.
 
\section*{Acknowledgments}
Much of this work in particular for the construction of the connected covering designs in Figures~\ref{fig:CC743},~\ref{fig:CC943},~\ref{fig:CC1143} strongly benefited from the La Jolla Covering Repository (\url{http://www.ccrwest.org/cover.html}) maintained by Dan Gordon.

\end{document}